\documentclass[12pt,reqno,twoside]{amsart}

\usepackage{amsfonts,amsmath,amssymb}
\usepackage{mathrsfs,mathtools}
\usepackage{enumerate}
\usepackage{hyperref}
\usepackage{esint}
\usepackage{graphicx}
\usepackage{bm}
\usepackage{commath}
\usepackage{esint}
\DeclareMathAlphabet{\mathpzc}{OT1}{pzc}{m}{it}

\usepackage{placeins} 
\usepackage{flafter} 

\usepackage{caption}

\numberwithin{equation}{section}

\usepackage[dvips,bottom=1.4in,right=1in,top=1in, left=1in]{geometry}

\makeatletter
\def\eqnarray{\stepcounter{equation}\let\@currentlabel=\theequation
\global\@eqnswtrue
\tabskip\@centering\let\\=\@eqncr
$$\halign to \displaywidth\bgroup\hfil\global\@eqcnt\z@
  $\displaystyle\tabskip\z@{##}$&\global\@eqcnt\@ne
  \hfil$\displaystyle{{}##{}}$\hfil
  &\global\@eqcnt\tw@ $\displaystyle{##}$\hfil
  \tabskip\@centering&\llap{##}\tabskip\z@\cr}

\def\endeqnarray{\@@eqncr\egroup
      \global\advance\c@equation\m@ne$$\global\@ignoretrue}

\newtheorem{theorem}{Theorem}[section]

\newtheorem{definition}[theorem]{Definition}

\newtheorem{lemma}[theorem]{Lemma}

\newtheorem{proposition}[theorem]{Proposition}
\newtheorem{assumption}[theorem]{Assumption}
\newtheorem{remark}[theorem]{Remark}
\numberwithin{equation}{section}

\def\Omc{\mathbb{R}^N\setminus\Omega}
\def\Omb{\mathbb{R}^N\setminus\overline{\Omega}}

\def\RR{{\mathbb{R}}}

\def\N{{\mathbb{N}}}

\def\Om{\Omega}

\def\pOm{\partial\Omega}
\def\RRext{\RR \cup \{+\infty\}}

\title{Exterior Nonlocal Variational Inequalities}

	\author{Harbir Antil}
\address{Harbir Antil, Department of Mathematical Sciences and the Center for Mathematics and Artificial Intelligence (CMAI), George Mason University,  Fairfax, VA 22030, USA.}
\email{hantil@gmu.edu}

\author{Madeline O. Horton}
\address{Madeline O. Horton, Department of Mathematical Sciences and the Center for Mathematics and Artificial Intelligence (CMAI), George Mason University,  Fairfax, VA 22030, USA.}
\email{mhorton5@gmu.edu}
\author{Mahamadi Warma}
\address{Mahamadi Warma, Department of Mathematical Sciences and the Center for Mathematics and Artificial Intelligence (CMAI), George Mason University,  Fairfax, VA 22030, USA.}
\email{mwarma@gmu.edu}

\thanks{This work is partially supported by NSF grant DMS-2110263, Air Force Office of Scientific Research under Award NO: FA9550-22-1-0248, Army Research Office (ARO) under Award NO: W911NF-20-1-0115, and Office of Naval Research (ONR) under Award NO: N00014-24-1-2147.}



\begin{document}

\begin{abstract}
This paper introduces a new class of variational inequalities where the obstacle 
is placed in the exterior domain that  
is disjoint from the observation domain. This is carried out with the help of 
nonlocal fractional operators. The need for such novel variational inequalities 
stems from the fact that the classical approach only allows placing the obstacle 
either inside the observation domain or on the boundary. A complete analysis 
of the continuous problem is provided. Additionally, perturbation arguments to 
approximate the problem are discussed. 
\end{abstract}

\maketitle

\section{Introduction}

Let $\Om \subset \RR^N$, $N \ge 1$, be a bounded open set with boundary $\pOm$. 
Moreover, let $\Sigma_1$, $\Sigma_2$ be open subsets of $\RR^N\setminus\overline{\Om}$
such that $\Sigma_1 \cap \Sigma_2 = \emptyset$ and $\overline\Sigma_1 \cup 
\overline\Sigma_2 = \Omc$. 
In this paper we introduce and study the following variational problem: 
Given 
$f \in W^{-s,2}(\Om,\Sigma_1)$, $z \in W^{s,2}(\Sigma_1)$, 
$\varphi \in W^{s,2}(\Sigma_2)$, we want to solve the following minimization problem:
    \begin{align}\label{eq:min_prob}
        \min_{u \in \mathcal{K}} J(u) 
            := \frac{C_{N,s}}{4}
               \int\int_{\RR^{2N}\setminus(\RR^N\setminus\Om)^2} &
               \frac{|u(x)-u(y)|^2}{|x-y|^{N+2s}} \;dxdy \notag\\
               &- \langle f , u \rangle_{W^{-s,2}(\Om,\Sigma_1),W^{s,2}(\Om,\Sigma_1)}   
    \end{align}
where $0<s<1$ is a real number, $\RR^{2N}\setminus(\RR^N\setminus\Om)^2
= (\Om\times\Om)\cup(\Om\times(\RR^N\setminus\Om))\cup((\RR^N\setminus\Om)\times\Om)$    
and the set of constraints is given by
 \[
    \mathcal{K} 
    := \{ u \in W^{s,2}(\RR^N) \, : \, u = z \mbox{ in } 
            \Sigma_1, \  u \le \varphi \mbox{ in } 
            \Sigma_2 \}.
 \]
Notice that $\mathcal K$ is a closed and convex subset of $W^{s,2}(\RR^N)$.    
The precise definition of the Sobolev spaces involved  will be given in Section~\ref{s:not}.

Obstacle and equilibrium problems, in general, have a rich history. They can capture many applications from phase changes to friction. Though in all cases either the obstacle is placed in the interior of $\Om$ or on the boundary $\pOm$. We refer to the monographs \cite{SBartels_2015a,DKindelehrer_GStampacchia_1980,rodrigues1987obstacle}. 

The main novelty of this paper is the introduction of the model \eqref{eq:min_prob}
which due to the presence of the nonlocal (fractional) Laplacian, enables placement of the obstacle $\varphi$
in the exterior of $\Omega$ and possibly disjoint from the boundary $\pOm$. After establishing existence of solutions (using standard arguments) to \eqref{eq:min_prob}, our first main result is given in Proposition~\ref{prop:equiv} which shows the equivalence between the variational problem \eqref{eq:min_prob} and three other characterizations:
\begin{itemize}
\item[(i)] variational inequality; 
\item[(ii)] slack variable (Lagrange multiplier) formulation, and;
\item[(iii)] weak formulation in the distributional sense.
\end{itemize}
Notice that similar results in the classical setting are well-known. Due to the nonlocal nature of the fractional Laplacian, the existing results do not directly extend to the fractional setting. Indeed, for instance, one has to carefully account for the nonlocal normal derivative. Tools from convex analysis such as tangent cone, convex indicator function are also employed to establish these results. Some of these arguments may appear to be standard, but the details are delicate due to the problem being nonlocal.

Our second main result corresponds to penalization of the constraints in the set $\mathcal{K}$ firstly in $L^2$-sense $\epsilon^{-2}\| (u-\varphi)^+ \|^2_{L^2(\Sigma_2)}$ (cf.~\eqref{eq:PL2Sig2}) and secondly in Sobolev 
sense $\xi^{-1}\| (u-\varphi)^+ \|_{W^{s,2}(\Sigma_2)}^2$ (cf.~\eqref{eq413}). Here $v^+ = \max\{v,0\}$. The former penalization is associated to the so-called Moreau-Yosida regularization. The latter has the distinct advantage of being able to provide a direct relationship between the slack variable for the original problem (ii) and it's penalized version. After establishing convergence using Mosco convergence arguments in Proposition~\ref{prop:mosco}, we establish a convergence rate in $\epsilon$ in Theorem~\ref{thm:penal}. We show that the penalized solution converges linearly in $\epsilon$ in $W^{s,2}(\Omega,\Sigma_1)$-norm. Moreover the constraint violation converges quadratically in $\epsilon$. Theorem~\ref{thm:zeta} provides linear convergence in $\xi$ for both the solution $u$ and the slack variable. 

Going forward, several of the techniques developed here will be helpful for the local problems and also in deriving finite element approximations. Notice that a popular way to numerically tackle these problems is to solve the penalized problems. Then the final approximation errors are governed by $\epsilon$ (or $\xi$) and the discretization errors. This will be part of a future investigation.
 
For completeness, we also mention that fractional obstacle problems where the obstacle is in the interior have also received a significant attention recently. See for instance \cite{BBarrios_AFigalli_XRos-Oton_2018a, LACaffarelli_SSalsa_LSilvestre_2008a,X_Ros-Oton_2018a}. However, as pointed out above, this is the first work that proposes to tackle the exterior obstacle problem.

\medskip
\noindent
{\bf Outline:}
The rest of the paper is organized as follows: In Section~\ref{s:not} we
first introduce some notations and state some preliminary results. Our main 
work starts from Section~\ref{s:VI_wellPosed}, where we establish several
equivalent formulations to the variational problem \eqref{eq:min_prob}. 
In Section \ref{Penalization} we consider two penalty approaches. In the 
first case we consider $L^2$-penalty and in the second case we consider 
a penalty in a Sobolev norm. Convergence and the precise rate of convergence with 
respect to the penalty parameters are established.

\section{Notation and Preliminaries}\label{s:not}

We begin this section by introducing some notations and give some preliminary results as they are needed throughout the paper.
Some of the results given in this section are well-known, in particular we follow
the notations from our previous works 
\cite{HAntil_RKhatri_MWarma_2018a, warma2018approximate}. 

\subsection{Fractional order Sobloev spaces and the fractional Laplacian}
For $\Om \subset\RR^N$ ($N \ge 1$) an arbitrary open set and $0 < s < 1$, we first define the classical Sobolev-Slobodecki space 
 \[
    W^{s,2}(\Om) := \left\{ u \in L^2(\Om) \;:\; 
            \int_\Om\int_\Om \frac{|u(x)-u(y)|^2}{|x-y|^{N+2s}}\;dxdy < \infty \right\} ,
 \]
and we endow it with the norm defined by

 \[
    \|u\|_{W^{s,2}(\Om)} := \left(\int_\Om |u|^2\;dx 
        + \int_\Om\int_\Om  \frac{|u(x)-u(y)|^2}{|x-y|^{N+2s}}\;dxdy \right)^{\frac12}.      
 \] 

Now, assume that the open set $\Omega$ is bounded. We also define the space 
 \[
    W^{s,2}(\Om,\Sigma_1) := \left\{ u \in W^{s,2}(\RR^N) \;:\; u = 0 \mbox{ in } 
            \Sigma_1 \right\} 
 \]     
which is a Hilbert space endowed with the norm induced by $W^{s,2}(\RR^N)$. 

Next, for
$u \in W^{s,2}(\Om,\Sigma_1)$, we let 
    \begin{equation}\label{NORM}
        \|u\|_{W^{s,2}(\Om,\Sigma_1)}
        = \left(\int\int_{\RR^{2N}\setminus(\RR^N\setminus\Om)^2} 
         \frac{|u(x)-u(y)|^2}{|x-y|^{N+2s}} \;dxdy\right)^{1/2}
    \end{equation}
where we recall that
$$\RR^{2N}\setminus(\RR^N\setminus\Om)^2
= (\Om\times\Om)\cup(\Om\times(\RR^N\setminus\Om))\cup((\RR^N\setminus\Om)\times\Om).$$

The following result is contained  in \cite[Proposition~5]{BAbdellaoui_ADieb_EValdinoci_2018a}.

    \begin{proposition}\label{prop:Gag}
        The norm $\|\cdot\|_{W^{s,2}(\Om,\Sigma_1)}$ given in \eqref{NORM} is equivalent to the one induced by 
        $W^{s,2}(\RR^N)$.  As a consequence, 
        $(W^{s,2}(\Om,\Sigma_1),\|\cdot\|_{W^{s,2}(\Omega,\Sigma_1)})$
        is a Hilbert space with the scalar product 
        \[
            (u,v)_{W^{s,2}(\Om,\Sigma_1)} 
            = \int\int_{\RR^{2N}\setminus(\RR^N\setminus\Om)^2} 
             \frac{(u(x)-u(y))(v(x)-v(y))}{|x-y|^{N+2s}} \;dxdy . 
        \]
    \end{proposition}

We denote the dual spaces of $W^{s,2}(\RR^N)$ and $W^{s,2}(\Om,\Sigma_1)$ 
by $W^{-s,2}(\RR^N)$ and $W^{-s,2}(\Om,\Sigma_1)$, respectively. 
Moreover, we will use $\langle\cdot,\cdot\rangle$, to denote their duality pairing 
whenever it is clear from the context. 

Next, we introduce the fractional Laplace operator. 
For $0<s<1$ we set 
\begin{equation*}
\mathbb{L}_s^{1}(\RR^N):=\left\{u:\RR^N\rightarrow
\mathbb{R}\;\mbox{
measurable, }\;\int_{\RR^N}\frac{|u(x)|}{(1+|x|)^{N+2s}}%
\;dx<\infty \right\} ,
\end{equation*}%
and for $u\in \mathbb{L}_s^{1}(\RR^N)$ and $\varepsilon >0$, we let
\begin{equation*}
(-\Delta )_{\varepsilon }^{s}u(x):=C_{N,s}\int_{\{y\in \RR^N,|y-x|>\varepsilon \}}
\frac{u(x)-u(y)}{|x-y|^{N+2s}}dy,\;\;x\in\RR^N,
\end{equation*}%
where the normalized constant $C_{N,s}$ is given by
\begin{equation}\label{CN}
C_{N,s}:=\frac{s2^{2s}\Gamma\left(\frac{2s+N}{2}\right)}{\pi^{\frac
N2}\Gamma(1-s)},
\end{equation}%
and $\Gamma $ is the usual Euler Gamma function (see, e.g. \cite%
{BCF,Caf3,Caf1,Caf2,NPV,War-DN1,War}). Then,  the fractional Laplacian 
$(-\Delta )^{s}$ is defined for $u\in \mathbb{L}_s^{1}(\RR^N)$  by the formula
\begin{align}
(-\Delta )^{s}u(x)=C_{N,s}\mbox{P.V.}\int_{\RR^N}\frac{u(x)-u(y)}{|x-y|^{N+2s}}dy 
=\lim_{\varepsilon \downarrow 0}(-\Delta )_{\varepsilon
}^{s}u(x),\;\;x\in\RR^N,\label{eq11}
\end{align}%
provided that the limit exists for a.e. $x\in\RR^N$. We refer to \cite{NPV} and the references therein for the class of functions for which the limit in \eqref{eq11} exists.

It has been shown in \cite[Proposition 2.2]{BPS} that for $u,v \in \mathcal{D}(\Om)$ (the space of all continuously infinitely differentiable functions with compact support in $\Omega$), we have 
that 
 \[
    \lim_{s\uparrow 1}\int_{\RR^N} v (-\Delta)^su\;dx 
        = - \int_{\RR^N} v \Delta u\;dx 
        = - \int_{\Om} v \Delta u\;dx=\int_{\Omega}\nabla u\cdot\nabla v\;dx.
 \]
This is where the constant $C_{N,s}$  given in \eqref{CN} plays a crucial role.

Next, we introduce the realization in $L^2(\Omega)$ of the operator $(-\Delta)^s$ with the mixed zero Dirichlet exterior condition in $\Sigma_1$ and zero nonlocal Neumann exterior condition in $\Sigma_2$. For this, consider the continuous, closed and coercive bilinear form $\mathcal E: W^{s,2}(\Om,\Sigma_1)\times  W^{s,2}(\Om,\Sigma_1)\to\RR$ given by
\begin{equation}\label{eq-BF}
\mathcal E(u,v):= \frac{C_{N,s}}{2} \int\int_{\RR^{2N}\setminus(\RR^N\setminus\Om)^2} 
         \frac{(u(x)-u(y))(v(x)-v(y))}{|x-y|^{N+2s}} \;dxdy,\;\; u,v\in W^{s,2}(\Om,\Sigma_1).
\end{equation}
Let $(-\Delta)_{\Sigma_1}^s$ be the self-adjoint operator in $L^2(\Omega)$ associated with $\mathcal E$ in the following sense:
\begin{equation}\label{def-op}
\begin{cases}
D((-\Delta)_{\Sigma_1}^s):=\left\{u\in  W^{s,2}(\Om,\Sigma_1),\; \exists f\in L^2(\Omega):\; \mathcal E(u,v)=(f,v)_{L^2(\Omega)}\;\forall v\in  W^{s,2}(\Om,\Sigma_1)\right\},\\
(-\Delta)_{\Sigma_1}^su=f\;\mbox{ in }\;\Omega.
\end{cases}
\end{equation}

\begin{remark}\label{rem-den}
{\em Since $C_c^\infty(\RR^N\setminus\Sigma_1)\subset D((-\Delta)_{\Sigma_1}^s)$, it is straightforward to show that $D((-\Delta)_{\Sigma_1}^s)$ is dense in $L^2(\Omega)$ and in $ W^{s,2}(\Om,\Sigma_1)$. 
}
\end{remark}

In the forthcoming discussion, we also make use of the local fractional order Sobolev space 
 \begin{equation}\label{eq:Ws2loc}
    W^{s,2}_{\rm loc}(\Omc) := \left\{ u \in L^2(\Omc) \;:\; u\varphi \in W^{s,2}(\Omc), 
         \ \forall \ \varphi \in \mathcal{D}(\Omc) \right\}.  
 \end{equation}

Furthermore, for $u \in W^{s,2}(\RR^N)$, using the terminology from \cite{SDipierro_XRosOton_EValdinoci_2017a}, we define the nonlocal normal derivative (or interaction operator) $\mathcal{N}_s$ as follows:
 \begin{align}\label{NLND}
    \mathcal{N}_s u(x) := C_{N,s} \int_\Om \frac{u(x)-u(y)}{|x-y|^{N+2s}}\;dy, 
            \quad x \in \RR^N \setminus \overline\Om . 
 \end{align}
Clearly $\mathcal{N}_s$ is a nonlocal operator and it is well defined on $W^{s,2}(\RR^N)$
as we discuss next (see e.g. \cite{warma2018approximate} for more details).

 \begin{lemma}\label{lem:Nmap}
  The interaction operator $\mathcal{N}_s$ maps continuously $W^{s,2}(\RR^N)$ into
 $ W^{s,2}_{\rm loc}(\RR^N\setminus\Om)$. 
  As a result, if  $u \in W^{s,2}(\RR^N)$, then $\mathcal{N}_s u \in L^2(\RR^N\setminus\Om)$.
 \end{lemma}

Despite the fact that $\mathcal{N}_s$ is defined on $\RR^N \setminus \overline{\Om}$, it
is still known as the ``normal" derivative. This is due to its similarity with 
the classical normal derivative, that is, it plays the same role for $(-\Delta)^s$ that the normal derivative does for the Laplace operator $-\Delta$ (see e.g. \cite[Proposition~2.2]{HAntil_RKhatri_MWarma_2018a}). 
Following the terminology from \cite{HAntil_RKhatri_MWarma_2018a}, we shall call 
$\mathcal N_s$ the {\em interaction operator} since it allows interaction 
between $\Omega$ and the exterior domain $\Omb$. 

We conclude this subsection by stating the integration by parts formulas for the fractional Laplacian, see \cite[Lemma 3.3]{SDipierro_XRosOton_EValdinoci_2017a} for smooth functions and 
\cite[Proposition 2.2]{HAntil_RKhatri_MWarma_2018a} for functions in Sobolev spaces (by using some density arguments). 
 \begin{proposition}[\bf The integration by parts formula I]
 \label{prop:prop}
 Let $u \in W^{s,2}(\RR^N)$ be such that $(-\Delta)^su \in L^2(\Omega)$. 
 Then,  for every $v\in W^{s,2}(\RR^N)$ we have that,
      \begin{align}\label{Int-Part1}
       \frac{C_{N,s}}{2} 
        \int\int_{\RR^{2N}\setminus(\RR^N\setminus\Om)^2} 
         \frac{(u(x)-u(y))(v(x)-v(y))}{|x-y|^{N+2s}} \;dxdy 
        = \int_\Om v(-\Delta)^s u\;dx + \int_{\Omc} v\mathcal{N}_s u\;dx.
      \end{align}     
 \end{proposition} 

We observe the following.

\begin{remark}
Let $(-\Delta)_{\Sigma_1}^s$ be the operator defined in \eqref{def-op}. Using the integration by parts formula \eqref{Int-Part1} we can deduce that
\begin{equation}\label{def-op-2}
\begin{cases}
D((-\Delta)_{\Sigma_1}^s)=\left\{u\in  W^{s,2}(\Om,\Sigma_1):\; \mathcal N_su=0\;\mbox{ in } \Sigma_2,\;(-\Delta)^s|_{\Omega}\in  L^2(\Om)\right\},\\
(-\Delta)_{\Sigma_1}^su=(-\Delta)^su\;\mbox{ in }\;\Omega.
\end{cases}
\end{equation}
\end{remark} 
 
The version of the integration-by-parts formula we will frequently use in this paper is the following 
(see e.g. \cite[Proposition~6]{BAbdellaoui_ADieb_EValdinoci_2018a}).
 \begin{proposition}[\bf The integration by parts formula II]
 \label{prop:prop_1}
 Let $u,v \in W^{s,2}(\Omega,\Sigma_1)$. 
 Then,
      \begin{align}\label{Int-Part}
       \frac{C_{N,s}}{2} 
        \int\int_{\RR^{2N}\setminus(\RR^N\setminus\Om)^2} &
         \frac{(u(x)-u(y))(v(x)-v(y))}{|x-y|^{N+2s}} \;dxdy \notag\\
        =& \langle (-\Delta)^s u,v\rangle_{W^{-s,2}(\Omega,\Sigma_1),W^{s,2}(\Omega,\Sigma_1)} + \int_{\Sigma_2} v\mathcal{N}_s u\;dx .
      \end{align}       
 \end{proposition} 

\begin{proof}
We notice that \eqref{Int-Part} has been stated differently in \cite[Proposition~6]{BAbdellaoui_ADieb_EValdinoci_2018a} without given a proof. They have replaced the duality map $\langle\cdot, \cdot\rangle$ with the scalar product $(\cdot,\cdot)_{L^2(\Omega)}$.
We think the right formulation is as given in \eqref{Int-Part}. For that reason we include the proof. We proceed in  two steps.

{\bf Step 1}: Firstly,  we observe that $(-\Delta)_{\Sigma_1}^s$ defined in \eqref{def-op} can be viewed as a bounded operator from $W^{s,2}(\Omega,\Sigma_1)$ into its dual $W^{-s,2}(\Omega,\Sigma_1)$ given by 
\begin{align*}
 \langle (-\Delta)_{\Sigma_1}^s u,v\rangle_{W^{-s,2}(\Omega,\Sigma_1),W^{s,2}(\Omega,\Sigma_1)}
 =\frac{C_{N,s}}{2} 
        \int\int_{\RR^{2N}\setminus(\RR^N\setminus\Om)^2} 
         \frac{(u(x)-u(y))(v(x)-v(y))}{|x-y|^{N+2s}} \;dxdy.
   \end{align*}
         
{\bf Step 2}: Let $u_n\in D((-\Delta)_{\Sigma_1}^s)$ be a sequence that converges to $u$ in  $W^{s,2}(\Omega,\Sigma_1)$, as $n\to\infty$. Existence of $u_n$ follows from Remark \ref{rem-den}.
It follows from Proposition \ref{prop:prop} that for every $n\in\mathbb N$ and $v\in W^{s,2}(\Omega,\Sigma_1)$, we have that,
\begin{align}\label{Int-seq}
       \frac{C_{N,s}}{2} 
        \int\int_{\RR^{2N}\setminus(\RR^N\setminus\Om)^2} &
         \frac{(u_n(x)-u_n(y))(v(x)-v(y))}{|x-y|^{N+2s}} \;dxdy \notag\\
        =& \int_\Om v(-\Delta)^s u_n\;dx + \int_{\Omc} v\mathcal{N}_s u_n\;dx \notag\\
        =&\langle (-\Delta)^s u_n,v\rangle_{W^{-s,2}(\Omega,\Sigma_1),W^{s,2}(\Omega,\Sigma_1)} + \int_{\Omc} v\mathcal{N}_s u_n\;dx .
      \end{align}
  Since $u_n$ converges to $u$  in $W^{s,2}(\Omega,\Sigma_1)$, as $n\to\infty$, it follows from Step 1 that $(-\Delta)^s u_n$ converges to $(-\Delta)^s u$ in $W^{-s,2}(\Omega,\Sigma_1)$, as $n\to\infty$.  It also follows from Lemma \ref{lem:Nmap} (the continuity of the operator $\mathcal N_s$) that $\mathcal N_su_n$ 
  converges to $\mathcal N_su$ in $L^2(\RR^N\setminus\Omega)$, as $n\to\infty$.  Finally since $v=0$ in $\Sigma_1$,  using all the above convergences and taking the limit of both sides of \eqref{Int-seq}, as $n\to\infty$, we get \eqref{Int-Part} and the proof is finished.
\end{proof}

\subsection{Useful results from convex analysis}

We will additionally require the following fundamental concepts from Convex Analysis. Consider a general problem of the form 
\begin{equation}\label{Gen}
    \min_{w \in W} f(w) \quad \mbox{subject to} \; G(w) \in \mathcal{K}_G, \; w \in \mathcal{C}
\end{equation}
where $W$ and $V$ are Banach spaces and $f: W \rightarrow \RR$, $G: W \rightarrow V$ are continuously Fr\'echet differentiable. Further, suppose that $\mathcal{C} \subset W$ is a non-empty, closed set and convex and $\mathcal{K}_G \subset V$ is a closed, convex cone. 

The feasible set is defined by 
\begin{align}\label{Gen Feasible}
    F := \{w \in W: G(w) \in \mathcal{K}_G, w \in \mathcal{C}\}.
\end{align}

Then, when $F \subset W$ is nonempty, we define the tangent cone of $F$ at $w \in F$ by 
\begin{equation}\label{Gen Tan Cone}
    T(F;w) := \{\tau \in W: \mbox{ for each } k \in\N, \exists r_k >0 \;  w_k \in F: \lim_{k \rightarrow \infty}w_k = w, \lim_{k \rightarrow \infty} r_k (w_k - w) = \tau\}
\end{equation}
and the linearization cone at a point $w \in F$ by 
\begin{equation}\label{Gen Lin Cone}
    L(F;w) := \{rh: r > 0, \ h\in W, \ G(w) + G'(w)h \in \mathcal{K}_G, \ w+h \in \mathcal{C}\}.
\end{equation}

For an optimal solution, $\overline{w}$ of \eqref{Gen}, it can be shown that the existence of Lagrange multipliers and construction of first order optimality conditions is dependent upon the linearization cone at $\overline{w} \in F$ to be contained in the tangent cone of $F$ at $\overline{w} \in F$. That is, $$L(F;\overline{w}) \subset T(F;\overline{w}).$$ This is sometimes referred to as the Abadie Constraint Qualification or just Constraint Qualification. A more detailed discussion of constraint qualifications has been addressed in \cite{ Guignard, Ulbrich_Hinze, Zowe_Kurcyusz}.

When formulating the Lagrangian in Section \ref{s:VI_wellPosed} it is necessary to introduce a few additional definitions. We refer to \cite{Eklend_Temam} for more details. As before, suppose that $W$ is a Banach space and let $W^*$ denote its topological dual with duality pairing $\langle \cdot, \cdot \rangle_{W^*, W}.$ Given $f: W \rightarrow \RR \cup \{+\infty\}$, its Fenchel conjugate is given by $f^*: W^* \rightarrow \RR \cup \{+\infty\}$, where
\begin{equation}\label{Fenchel1}
    f^*(\lambda) := \sup_{w \in W} \{\langle \lambda, w \rangle_{W^*, W} - f(w)\}.
\end{equation}

Additionally, for a convex set $\mathcal{Z} \subset W$ we define the indicator functional of $\mathcal{Z}$ by $I_{\mathcal{Z}}: W \rightarrow \RR \cup \{+\infty\}$, 
\begin{equation}\label{Indicator Func}
         I_{\mathcal{Z}}(u) = \begin{cases}
          0 & \mbox{if} \; u \in \mathcal Z \\
          +\infty & \mbox{if} \; u \notin \mathcal Z.
      \end{cases}
\end{equation}

Then, in light of \eqref{Fenchel1} and \eqref{Indicator Func} we have the following result.

\begin{lemma} \label{Lemma:dual_ind}
    Consider the sets $\mathcal{Z}^- = \{w \in W: w \leq 0\}$ and $\mathcal{Z}^+ = \{\eta \in W^*: \eta \geq 0\}.$ If $W$ is reflexive, then $I^*_{\mathcal{Z^-}}(\lambda) = I_{\mathcal{Z^+}}(\lambda)$ and  $I^*_{Z^+}(u) = I_{\mathcal{Z^-}}(u).$
\end{lemma}

\begin{proof} Notice that the Fenchel conjugate of $I_{\mathcal{Z}^-}$, is given by
  $I_{\mathcal{Z}^-}^*: W^* \rightarrow \RRext$ where,
  \begin{align*}
      I_{\mathcal{Z}^-}^*(\lambda) = &\sup_{w \in W} \{\langle \lambda, w\rangle_{W^*,W} - I_{\mathcal{Z}^-}(w)\} \\
      = & \sup_{v \in \mathcal{Z}^-} \langle \lambda, w \rangle_{W^*,W} \\
      =& I_{Z^+} (\lambda) \, .
  \end{align*}  
  The second equality follows in a similar fashion. 
The proof is finished.
  \end{proof}

\section{Well-posedness of the variational inequality}
\label{s:VI_wellPosed}

We begin this section by introducing the notion of solutions to the minimization problem \eqref{eq:min_prob}. 
Throughout the rest of the paper $\Omega$, $\Sigma_1$ and $\Sigma_2$ are as in the introduction. We also assume the following.

\begin{assumption}
We assume that $\Sigma_1$ has the extension property in the sense that for every $z\in W^{s,2}(\Sigma_1)$, there exists a function $\mathcal Z\in W^{s,2}(\RR^N)$ such that $\mathcal Z|_{\Sigma_1}=z$.
\end{assumption}

Next, we give our notion of solutions to  \eqref{eq:min_prob}.

    \begin{definition}\label{def:solnot}
        Given $f \in W^{-s,2}(\Om,\Sigma_1)$, $z \in W^{s,2}(\Sigma_1)$, 
        $\varphi \in W^{s,2}(\Sigma_2)$, let $\mathcal{Z} \in W^{s,2}(\RR^N)$ be such 
        that $\mathcal{Z}|_{\Sigma_1} = z$.         
        Then, $u \in \mathcal{K}$ solves \eqref{eq:min_prob},  if 
        $u-\mathcal{Z} \in \mathcal{K}_0$ solves the minimization problem 
        \begin{equation}\label{eq:min_prob_1}
            \min_{u -\mathcal{Z} \in \mathcal{K}_0} J(u), 
        \end{equation}
        where     
        $$
            \mathcal{K}_0 
            := \{ w \in W^{s,2}(\RR^N) \, : \, w = 0 \mbox{ in } 
                \Sigma_1, \  w \le \varphi \mbox{ in } 
                \Sigma_2 \}  
        $$ 
    and we recall that
    $$\mathcal K=\{u\in W^{s,2}(\RR^N):\; u=z\mbox{ in }\Sigma_1,\; u\le\varphi\;\mbox{ in }\Sigma_2\}.$$
    \end{definition}

Notice that $\mathcal{K}$ and $\mathcal{K}_0$ only differs by the fact that functions in 
$\mathcal{K}_0$ are zero in $\Sigma_1$. 
The next result states the well-posedness of the minimization problem \eqref{eq:min_prob}
according to Definition~\ref{def:solnot}. 

\begin{theorem}
        Let $f \in W^{-s,2}(\Om,\Sigma_1)$, $z \in W^{s,2}(\Sigma_1)$ and 
        $\varphi \in W^{s,2}(\Sigma_2)$.
        Then, there exists a unique solution $u \in \mathcal{K}$ to the minimization problem 
        \eqref{eq:min_prob} according to Definition~\ref{def:solnot}. 
    \end{theorem}
    
    \begin{proof}
        Let $u\in\mathcal K$ and set $w:=u-\mathcal{Z}$. Notice that $w|_{\Sigma_1} = 0$. Then,  using the definition of the functional
        $J$, the minimization problem \eqref{eq:min_prob_1} can be rewritten as follows:
        \begin{equation}\label{eq:min_prob_12}
            \min_{w \in \mathcal{K}_0} J(w) :=\frac{C_{N,s}}{4} 
               \int\int_{\RR^{2N}\setminus(\RR^N\setminus\Om)^2} 
               \frac{|w(x)-w(y)|^2}{|x-y|^{N+2s}} \;dxdy - \langle f, w \rangle_{W^{-s,2}(\Om,\Sigma_1),W^{s,2}(\Om,\Sigma_1)} .
        \end{equation}
        We recall from Proposition~\ref{prop:Gag} that the norm
        \[
        \|u\|_{W^{s,2}(\Om,\Sigma_1)} 
        = \left(\int\int_{\RR^{2N}\setminus(\RR^N\setminus\Om)^2} 
         \frac{|u(x)-u(y)|^2}{|x-y|^{N+2s}} \;dxdy\right)^{\frac 12}
        \]
        is equivalent to the one induced by the space $W^{s,2}(\RR^N)$. Since $\mathcal{K}_0$ is
        closed and convex,  we have that the existence of solutions to the minimization problem \eqref{eq:min_prob_12}
        follows from the direct
        method of the calculus of variations. Uniqueness is due to the fact that the functional $J$ 
        is strictly convex.  The proof is finished.      
    \end{proof}

Throughout the remainder of the paper $\mathcal E$ denotes the bilinear form with domain $D(\mathcal E)=W^{s,2}(\Om,\Sigma_1)$ defined in \eqref{eq-BF}.

We have the following important result on various equivalent formulations to the minimization problem \eqref{eq:min_prob}, hence to \eqref{eq:min_prob_1}.
    \begin{proposition}\label{prop:equiv}
    The following assertions hold.
    \begin{enumerate}
   \item  A function $u\in\mathcal K$ solves the minimization problem \eqref{eq:min_prob} if and only if $w:=u-\mathcal Z\in \mathcal K_0$ satisfies the variational inequality
         \begin{equation}\label{eq: VI1}
      \mathcal{E}(w, v-w) - \langle f, v-w \rangle_{W^{-s,2}(\Om,\Sigma_1),W^{s,2}(\Om,\Sigma_1)} \geq 0, \;\forall v \in \mathcal K_0.
  \end{equation}
        
    \item The variational inequality \eqref{eq: VI1} is equivalent to 
        the following: There exists a non negative functional $\lambda\in W^{-s,2}(\Sigma_2)$ such that, 
    \begin{align}\label{eq: KKT}
      \begin{cases}
         \mathcal{E}(w,v) + \langle \lambda, v \rangle_{W^{-s,2}(\Sigma_2), W^{s,2}(\Sigma_2)} = \langle f,v \rangle_{W^{-s,2}(\Om,\Sigma_1),W^{s,2}(\Om,\Sigma_1)} & \mbox{for all} \; v \in W^{s,2}(\Omega, \Sigma_1) \\
         w \leq \varphi & \mbox{in} \; \Sigma_2 \\
         \langle \lambda, \Tilde{v} \rangle_{W^{-s,2}(\Sigma_2), W^{s,2}(\Sigma_2)} \leq 0 & \mbox{for all} \; \Tilde{v} \in \mathcal{K}_0^- \\
         \langle \lambda, w - \varphi \rangle_{W^{-s,2}(\Sigma_2), W^{s,2}(\Sigma_2)} = 0,
      \end{cases}
  \end{align}    
  where
    \begin{equation}\label{K0}
        \mathcal{K}_0^- = \{\tilde v \in W^{s,2}(\Sigma_2): \tilde v \leq 0\}.
         \end{equation}
         
    \item Additionally, the variational inequality \eqref{eq: VI1} is equivalent to the Euler-Lagrange equations
    \begin{align}\label{EL1}
        \begin{cases}
            (-\Delta)^s w = f & \mbox{in} \; \mathcal{D}(\Omega)^*, \\
            (-\Delta)^s w = 0 & \mbox{in} \; \mathcal{D}(\Sigma_2)^*, \\
                  \mathcal{N}_s w \leq 0 & \mbox{in} \; \Sigma_2, \\
            \mathcal{N}_s w = 0 & \mbox{in} \; \Sigma_2 \cap \{u < \varphi\}, \\
            u \leq \varphi & \mbox{in} \; \Sigma_2.
        \end{cases}
    \end{align}
   The last two conditions of \eqref{EL1} are also equivalent to the complimentarity condition 
    \begin{align}\label{CC}
        (u - \varphi) \mathcal{N}_s w = 0 \quad \mbox{in} \; \Sigma_2.
    \end{align}
   \end{enumerate}  
    \end{proposition}

     \begin{proof}
 We proceed in three steps.

 {\bf Step 1}: Let $u\in\mathcal K$ solve the minimization problem \eqref{eq:min_prob}. Then, by Definition \ref{def:solnot}
$w:=u - \mathcal{Z}  \in \mathcal{K}_0$ solves (\ref{eq:min_prob_12}). From the convexity of $\mathcal{K}_0,$ for all $v \in \mathcal{K}_0$ and all $t \in [0,1],$ we know that $w + t(v-w) \in \mathcal{K}_0.$ As a result, for $t \in (0,1]$, we have that
  \begin{align*}
      0 \leq& \frac{1}{t}\bigg(J(w + t(v-w)) - J(w)\bigg)  \\
     = &\frac{C_{N,s}}{4t} \int\int_{\RR^{2N} \backslash (\RR^N \backslash \Omega)^2} \frac{|(w + t(v-w))(x) - (w + t(v-w))(y)|^2}{|x-y|^{N + 2s}} \ dxdy\\
     &-  \frac{C_{N,s}}{4t} \int\int_{\RR^{2N} \backslash (\RR^N \backslash \Omega)^2} \frac{|w(x) - w(y)|^2}{|x-y|^{N + 2s}} \ dxdy - \langle f, v-w \rangle_{W^{-s,2}(\Om,\Sigma_1),W^{s,2}(\Om,\Sigma_1)} \\
     =&  \frac{C_{N,s}}{2}\int\int_{\RR^{2N} \backslash (\RR^N \backslash \Omega)^2} \frac{(w(x) - w(y))((v-w)(x) - (v-w)(y)}{|x-y|^{N+2s}} \ dxdy\\
     &- \langle f, v-w\rangle_{W^{-s,2}(\Om,\Sigma_1),W^{s,2}(\Om,\Sigma_1)}\\
      &+ \frac{C_{N,s} t}{4} \int\int_{\RR^{2N} \backslash (\RR^N \backslash \Omega)^2} \frac{|(v-w)(x) - (v-w)(y)|^2}{|x-y|^{N+2s}} \ dxdy.
  \end{align*}
  Taking the limit of the preceding inequality, as $t \downarrow 0$, we get the variational inequality \eqref{eq: VI1}.

  Conversely, if $u - \mathcal{Z} = w \in \mathcal{K}_0$ satisfies \eqref{eq: VI1}, then for any $v \in \mathcal{K}_0$, we have that
  \begin{align*}
      J(v) =& J(w + (v-w)) \\
      = & \frac{C_{N,s}}{4}\int\int_{\RR^{2N} \backslash (\RR^N \backslash \Omega)^2} \frac{|w(x) - w(y)|^2}{|x-y|^{N+2s}} \ dxdy - \langle f, w \rangle_{W^{-s,2}(\Om,\Sigma_1),W^{s,2}(\Om,\Sigma_1)} \\
      &+ \frac{C_{N,s}}{2}\int\int_{\RR^{2N} \backslash (\RR^N \backslash \Omega)^2} \frac{(w(x) - w(y))((v-w)(x) - (v-w)(y))}{|x-y|^{N+2s}} \ dxdy \\
      &- \langle f, v - w \rangle_{W^{-s,2}(\Om,\Sigma_1),W^{s,2}(\Om,\Sigma_1)} \\
      &+ \frac{C_{N,s}}{4} \int\int_{\RR^{2N} \backslash (\RR^N \backslash \Omega)^2} \frac{|(v-w)(x) - (v-w)(y)|^2}{|x-y|^{N+2s}} \ dxdy \\
      \geq & J(w).
  \end{align*}
  Hence, $w\in\mathcal K_0$ is a minimizer of $J$ and we have shown that \eqref{eq: VI1} implies \eqref{eq:min_prob}. The proof of Part (a) is complete.
  
{\bf Step 2}: Now, let us consider another characterization of the solution to (\ref{eq:min_prob_12}). In particular, we are able to include the constraint $w \in \mathcal{K}_0$ into the minimization problem if we instead consider the functional 
  \begin{align}\label{eq:form_uncon_func}
      \Tilde{J}(w) =& \frac{C_{N,s}}{4} \int\int_{\RR^{2N} \backslash (\RR \backslash \Omega)^2} \frac{|w(x) - w(y)|^2}{|x-y|^{N+2s}} \ dxdy \notag\\
      &- \langle f, w \rangle_{W^{-s,2}(\Om,\Sigma_1),W^{s,2}(\Om,\Sigma_1)} + I_{\mathcal{K}_0^-}(w-\varphi),
  \end{align} 
  where  $\mathcal K_0^-$ is given in \eqref{K0} and $I_{\mathcal{K}_0^-}$ is the indicator function defined in \eqref{Indicator Func}.   
It follows from Lemma \ref{Lemma:dual_ind} that, 
  \begin{align*}
      I_{\mathbb K_0^+}^*(u) = I_{\mathcal{K}_0^{-}}(u)
  \end{align*}
  where $\mathbb K_0^+ = \{\eta \in W^{-s,2}(\Sigma_2): \eta \geq 0\}.$
  In the definition of $\mathbb K_0^+$,  by $\eta\ge 0$ we mean that $\langle\eta,v\rangle_{W^{-s,2}(\Sigma_2),W^{s,2})\Sigma_2)}\ge 0$ for all $v\in W^{s,2}(\Sigma_2)$ with $v\ge 0$ a.e. in $\Sigma_2$.
  Therefore, 
  \begin{align}
      \notag \inf_{w \in W^{s,2}(\Omega, \Sigma_1)} \Tilde{J}(w) 
     =& \notag \inf_{w \in W^{s,2}(\Omega, \Sigma_1)} \bigg(\frac{C_{N,s}}{4} \int\int_{\RR^{2N} \backslash (\RR \backslash \Omega)^2} \frac{|w(x) - w(y)|^2}{|x-y|^{N+2s}} \ dxdy\notag\\
     &- \langle f,w \rangle_{W^{-s,2}(\Om,\Sigma_1),W^{s,2}(\Om,\Sigma_1)} + I_{\mathcal{K}_0^-} (w-\varphi)\bigg) \notag\\
      = &\inf_{w \in W^{s,2}(\Omega, \Sigma_1)} \bigg(\frac{C_{N,s}}{4} \int\int_{\RR^{2N} \backslash (\RR \backslash \Omega)^2} \frac{|w(x) - w(y)|^2}{|x-y|^{N+2s}} \ dxdy\notag\\
      &- \langle f,w \rangle_{W^{-s,2}(\Om,\Sigma_1),W^{s,2}(\Om,\Sigma_1)} + I_{\mathbb K_0^+}^* (w-\varphi)\bigg) \label{IK01}.
    \end{align}
Applying the definition of $I^*_{\mathbb K_0^+}$ further shows that  the right hand side of \eqref{IK01} becomes 
 \begin{multline}\label{IK02}
      \inf_{w \in W^{s,2}(\Omega, \Sigma_1)} \bigg(\frac{C_{N,s}}{4} \int\int_{\RR^{2N} \backslash (\RR \backslash \Omega)^2} \frac{|w(x) - w(y)|^2}{|x-y|^{N+2s}} \ dxdy - \langle f,w \rangle_{W^{-s,2}(\Om,\Sigma_1),W^{s,2}(\Om,\Sigma_1)}  \\
      + \sup_{\lambda \in W^{-s,2}(\Sigma_2)} \{\langle \lambda, w - \varphi \rangle_{W^{-s,2}(\Sigma_2),W^{s,2}(\Sigma_2)} - I_{\mathbb K_0^+}(\lambda)\}\bigg).
\end{multline}

Since the supremum in \eqref{IK02} can only be reached when $\lambda \in \mathbb K_0^+$, we can deduce that 
\begin{align}
      & \notag\inf_{w \in W^{s,2}(\Omega, \Sigma_1)} \bigg(\frac{C_{N,s}}{4} \int\int_{\RR^{2N} \backslash (\RR \backslash \Omega)^2} \frac{|w(x) - w(y)|^2}{|x-y|^{N+2s}} \ dxdy\notag\\
      &- \langle f,w \rangle_{W^{-s,2}(\Om,\Sigma_1),W^{s,2}(\Om,\Sigma_1)}\notag+ \sup_{\lambda \in \mathbb K_0^+} \langle \lambda, w - \varphi \rangle_{W^{-s,2}(\Sigma_2),W^{s,2}(\Sigma_2)} \bigg) \\
      = &\inf_{w \in W^{s,2}(\Omega, \Sigma_1)}\sup_{\lambda \in \mathbb K_0^+} \bigg(\frac{C_{N,s}}{4} \int\int_{\RR^{2N} \backslash (\RR \backslash \Omega)^2} \frac{|w(x) - w(y)|^2}{|x-y|^{N+2s}} \ dxdy \notag\\
      &- \langle f,w \rangle_{W^{-s,2}(\Om,\Sigma_1),W^{s,2}(\Om,\Sigma_1)} +  \langle \lambda, w - \varphi \rangle_{W^{-s,2}(\Sigma_2),W^{s,2}(\Sigma_2)} \bigg).
  \end{align}
The above identities motivate an associated Lagrangian, given by 
  \begin{align}\label{Lagrange1}
      \mathcal{L}(u,\eta) :=& \frac{C_{N,s}}{4} \int\int_{\RR^{2N} \backslash (\RR \backslash \Omega)^2} \frac{|u(x) - u(y)|^2}{|x-y|^{N+2s}} \ dxdy \notag\\
      &- \langle f,u \rangle_{W^{-s,2}(\Om,\Sigma_1),W^{s,2}(\Om,\Sigma_1)} + \langle \eta, u - \varphi \rangle_{W^{-s,2}(\Sigma_2), W^{s,2}(\Sigma_2)}.
  \end{align}

  Next, we define the tangent cone of $\mathcal{K}_0$ at $u \in \mathcal{K}_0$ by 
  \begin{gather*}
      T(\mathcal{K}_0; u) := \Big\{\kappa \in W^{s,2}(\Omega, \Sigma_1): \; \mbox{for each} \; k \in \N, \; \exists r_k > 0, u_k \in \mathcal{K}_0\\
      : \lim_{k \rightarrow \infty} u_k = u, \lim_{k \rightarrow \infty} r_k(u_k - u) = \kappa \Big\} 
  \end{gather*}
  and the linearization cone at $u \in \mathcal K_0$ by 
  \begin{gather*}
      L(\mathcal{K}_0; u) := \Big\{rh: r>0, h \in W^{s,2}(\Omega, \Sigma_1), u+h - \varphi \in \mathcal{K}_0^-\Big\}.
  \end{gather*}
 Notice that, whenever $w \in \mathcal{K}_0$ solves (\ref{eq:min_prob_12}) we have that $L(\mathcal{K}_0; w) \subset T(\mathcal{K}_0; w).$ Indeed, suppose that $\kappa \in L(\mathcal{K}_0;w).$ Then, for some $r > 0$ and $h \in W^{s,2}(\Omega, \Sigma_1)$ we have that $\kappa = rh$ where $w + h \leq \varphi$ in $\Sigma_2.$ It follows from the convexity of $\mathcal{K}_0$ that $w + \frac{1}{k}h \in \mathcal{K}_0$ for any $k \in \N.$ Then,  choosing 
      $w_k = w + \frac{1}{k}h$ and
      $r_k = kr$, 
 we have that $\lim_{k \rightarrow \infty} w_k = w$ and 
  \begin{align*}
      \lim_{k \rightarrow \infty} r_k(w_k - w) = \lim_{k \rightarrow \infty} kr(w + \frac{1}{k}h - w) = \kappa
  \end{align*}
  so that $\kappa \in T(\mathcal{K}_0;w).$ 
  
  Then, it is well-known (see e.g. \cite{Ulbrich_Hinze}) that there exists a Lagrange multiplier, $\lambda \in W^{-s,2}(\Sigma_2)$ so that $(w, \lambda)$ satisfies the KKT conditions \eqref{eq: KKT}. It follows from the third condition in \eqref{eq: KKT} that $\lambda$ is non-negative in the sense that $\langle\lambda,\tilde v\rangle_{W^{-s,2}(\Sigma_2),W^{s,2}(\Sigma_2)}\ge 0$ for every $\tilde v\in W^{s,2}(\Sigma_2)$, $\tilde v\ge 0$ a. e. in $\Sigma_2$.

Conversely, taking $v:=v-w$ in the first identity in \eqref{eq: KKT} with $v\in\mathcal K_0$, we get that for every $v\in \mathcal K_0$,
\begin{align}\label{A1}
   \mathcal{E}(w,v-w) &-\langle f,v-w\rangle_{W^{-s,2}(\Omega,\Sigma_1),W^{s,2}(\Omega,\Sigma_1)}\notag\\
   =&- \langle \lambda, v -w\rangle_{W^{-s,2}(\Sigma_2), W^{s,2}(\Sigma_2)}\notag\\
   =&- \langle \lambda, v-\varphi+\varphi -w\rangle_{W^{-s,2}(\Sigma_2), W^{s,2}(\Sigma_2)}\notag\\
 =&- \langle \lambda, v-\varphi\rangle_{W^{-s,2}(\Sigma_2), W^{s,2}(\Sigma_2)} +  \langle \lambda, w-\varphi \rangle_{W^{-s,2}(\Sigma_2), W^{s,2}(\Sigma_2)}.
\end{align}
It follows from the last identity in \eqref{eq: KKT} that
\begin{equation}\label{A2}
 \langle \lambda, w-\varphi \rangle_{W^{-s,2}(\Sigma_2), W^{s,2}(\Sigma_2)}=0.
\end{equation}
Since $v-\varphi\in\mathcal K_0^-$, it follows from the third inequality in \eqref{eq: KKT} that 
\begin{equation}\label{A3}
- \langle \lambda, v-\varphi\rangle_{W^{-s,2}(\Sigma_2), W^{s,2}(\Sigma_2)}\ge 0.
\end{equation}
Combining \eqref{A1}, \eqref{A2} and \eqref{A3} we can deduce that
\begin{equation*}
\mathcal{E}(w,v-w) -\langle f,v-w\rangle_{W^{-s,2}(\Omega,\Sigma_1),W^{s,2}(\Omega,\Sigma_1)}\ge 0,
\end{equation*}
and we have shown \eqref{eq: VI1}.
  
 {\bf Step 3}: It remains to show Part (c). Suppose that $w: = u - \mathcal Z \in \mathcal{K}_0$ solves \eqref{eq: VI1}.
 Applying the integration by parts formula given in (\ref{Int-Part}), we can rewrite the variational inequality \eqref{eq: VI1} as follows: For all $v \in \mathcal{K}_0$,
  \begin{equation}\label{eq:VI2}
      \langle (-\Delta)^s w, v-w \rangle_{W^{-s,2}(\Omega, \Sigma_1), W^{s,2}(\Omega, \Sigma_1)} + \int_{\Sigma_2} (v-w) \mathcal{N}_s w \ dx \geq \langle f, v-w \rangle_{W^{-s,2}(\Omega, \Sigma_1), W^{s,2}(\Omega, \Sigma_1)}.
  \end{equation}
 
  Let $\zeta \in \mathcal{D}(\Omega)$ be arbitrary. It is clear that $w + \zeta \in \mathcal{K}_0$, so setting $v: = w + \zeta$ in \eqref{eq:VI2} yields 
  \begin{align*}
      \langle (-\Delta)^s w - f,   \zeta \rangle_{W^{-s,2}(\Omega, \Sigma_1), W^{s,2}(\Omega, \Sigma_1)} \geq 0.
  \end{align*}
  Since this is also true for $-\zeta$, we can deduce that
  \begin{equation}
  \langle (-\Delta)^s w - f,   \zeta \rangle_{W^{-s,2}(\Omega, \Sigma_1), W^{s,2}(\Omega, \Sigma_1)} = 0 \quad \mbox{for all} \; \zeta \in \mathcal{D}(\Omega).
  \end{equation}
  That is, $(-\Delta)^s w = f$ in $\mathcal{D}(\Omega)^*.$
  
Now, suppose that $\psi \in \mathcal{D}(\Sigma_2).$ Then, 
    \begin{align*}
        \mathcal{E}(w, \psi) &=\frac{C_{N,s}}{2} \int\int_{\RR^{2N}\backslash (\RR^N \backslash \Omega)^2} \frac{(w(x) - w(y))(\psi(x)-\psi(y))}{|x-y|^{N+2s}} \ dxdy \\
        & =\frac{C_{N,s}}{2}\bigg( - \int_{\Sigma_2}\int_{\Omega} \frac{\psi(y)(w(x) - w(y))}{|x-y|^{N+2s}} \ dxdy + \int_{\Omega} \int_{\Sigma_2} \frac{\psi(x) (w(x) - w(y))}{|x-y|^{N+2s}} \ dxdy\bigg) \\
        & = C_{N,s} \int_{\Sigma_2} \psi(x)\bigg(\int_{\Omega} \frac{w(x) - w(y)}{|x-y|^{N+2s}} \ dy \bigg) \ dx\\
        &= \int_{\Sigma_2} \psi(x) \mathcal{N}_sw(x) \ dx.
    \end{align*}
 Combining this with (\ref{Int-Part}) shows that 
    \begin{equation}\label{EqinSigma2}
        \langle (-\Delta)^sw, \psi \rangle_{W^{-s,2}(\Om,\Sigma_1),W^{s,2}(\Om,\Sigma_1)} = 0 \qquad \mbox{for all} \; \psi \in \mathcal{D}(\Sigma_2)
    \end{equation}
    so that $(-\Delta)^s w = 0$ in $\mathcal{D}(\Sigma_2)^*.$

  Now, consider the set $E := \{x \in \Sigma_2: w(x) < \varphi(x)\}.$ Suppose that $\psi \in\mathcal{D}(E),$ the space of test functions in $E.$ For sufficiently small $\epsilon > 0$ we have that $v: = w + \epsilon \psi \in \mathcal{K}_0.$ Our variational inequality, along with (\ref{EqinSigma2}) yields 
  \begin{align*}
      \epsilon \int_{\Sigma_2} \psi \mathcal{N}_sw \ dx \geq \epsilon \langle f, \psi \rangle_{W^{-s,2}(\Omega, \Sigma_1), W^{s.2}(\Omega, \Sigma_1)}.
  \end{align*}
  This is also true for $-\psi$ so, 
  \begin{equation}\label{ss}
      \int_{\Sigma_2} \psi \mathcal{N}_s w \ dx = \langle f, \psi \rangle_{W^{-s,2}(\Omega, \Sigma_1), W^{s.2}(\Omega, \Sigma_1)} \qquad \mbox{for all} \; \psi \in \mathcal{D}(E).
  \end{equation}

Since $f$ is a continuous linear functional on $W^{s,2}(\Omega, \Sigma_1)$, from the Riesz representation theorem there exists a unique $\Tilde{f} \in W^{s,2}(\Omega, \Sigma_1)$ such that for every $\psi\in\mathcal D(E)$,
    \begin{align}\label{D1}
      \langle f, \psi \rangle_{W^{-s,2}(\Omega, \Sigma_1), W^{s,2}(\Omega, \Sigma_1)} = &(\tilde f,\psi)_{W^{s,2}(\Omega,\Sigma_1)}\notag\\
      =&
      \int\int_{\RR^{2N}\backslash (\RR^N \backslash \Omega)^2} \frac{(\Tilde{f}(x) - \Tilde{f}(y))(\psi(x) - \psi(y))}{|x-y|^{N+2s}} \ dxdy.
  \end{align}
 Since $\psi = 0$ outside of $\Sigma_2$ and  $\RR^{2N}\backslash (\RR^N \backslash \Omega)^2 = (\Omega \times \Omega) \cup ( \Omega \times \RR^N \backslash \Omega)) \cup ((\RR^N \backslash \Omega) \times \Omega)$, we have that the identity \eqref{D1} reduces to the following:
  \begin{align}\label{mm}
  & \int\int_{\RR^{2N}\backslash (\RR^N \backslash \Omega)^2} \frac{(\Tilde{f}(x) - \Tilde{f}(y))(\psi(x) - \psi(y))}{|x-y|^{N+2s}} \ dxdy\notag\\
      =&\int_{\Omega} \int_{\Sigma_2} \frac{\psi(x)(\Tilde{f}(x) - \Tilde{f}(y))}{|x-y|^{N+2s}} \ dxdy + \int_{\Sigma_2} \int_{\Omega} \frac{-\psi(y) (\Tilde{f}(x) - \Tilde{f}(y))}{|x-y|^{N+2s}} \ dxdy \notag \\
      = &\int_{\Sigma_2}\int_{\Omega} \frac{\psi(x)(\Tilde{f}(x) - \Tilde{f}(y))}{|x-y|^{N+2s}} \ dy dx + \int_{\Sigma_2}\int_{\Omega} \frac{-\psi(x)(\Tilde{f}(y) - \Tilde{f}(x))}{|x-y|^{N+2s}} \ dydx \notag \\
       =&2 \int_{\Sigma_2}\int_{\Omega} \frac{\psi(x)(\Tilde{f}(x) - \Tilde{f}(y))}{|x-y|^{N+2s}} \ dydx\notag\\
      =&2\int_{\Sigma_2}\psi(x)\mathcal N_s(\tilde f)(x)\;dx.
  \end{align}
It follows from \eqref{ss}, \eqref{D1}  and \eqref{mm} that
  \begin{align}
      \int_{\Sigma_2} \psi(x) \bigg(\mathcal{N}_s w(x) - 2\mathcal N_s(\tilde f)(x)\bigg) dx = 0
  \end{align}
  for all $\psi \in \mathcal{D}(E).$ We can deduce from the fundamental lemma of calculus of variations that
  \begin{equation}\label{D2}
      \mathcal{N}_s w = 2 \mathcal N_s(\tilde f)\;\mbox{ in } \Sigma_2,
  \end{equation}
  whenever $w< \varphi$ in $\Sigma_2.$
  
 Next, since $W^{s,2}(\Omega, \Sigma_1)$ is a linear subspace of $L^2(\Omega)$ it follows from the Hahn Banach Theorem that, there exists a linear functional $\hat{f}\in (L^2(\Omega))^\star=L^2(\Omega)$ such that for every $\psi\in\mathcal D(E)$,
 \begin{equation}\label{jj}
  \langle f, \psi \rangle_{W^{-s,2}(\Omega, \Sigma_1), W^{s,2}(\Omega, \Sigma_1)}=(\hat f,\psi)_{L^2(\Omega)}  .  
 \end{equation}
 It follows from \eqref{D1}, \eqref{mm} and \eqref{jj} that for every $\psi\in\mathcal D(E)$, 
 \begin{align*}
  2\int_{\Sigma_2}\psi(x)\mathcal N_s(\tilde f)(x)\;dx=\int_{\Omega}\psi(x)\hat f(x)\;dx=0   
 \end{align*}
 where we have also used that $\operatorname{supp}[\psi]\subset \Sigma_2$. It follows from the fundamental lemma of the calculus of variation that
 $$\mathcal N_s(\tilde f)=0\;\;\mbox{ in }\Sigma_2.$$
This fact together with \eqref{D2} implies that
$$\mathcal N_sw=0\;\;\mbox{ in }\{x \in \Sigma_2: u(x) < \varphi(x)\}.$$

  Now, suppose that $\psi \in \mathcal{D}(\Sigma_2)$ with $\psi \geq 0$. Substituting $v: = w - \psi$ into (\ref{eq: VI1}), we can see that 
  \begin{equation}\label{eq:PP}
      \int_{\Sigma_2} \psi \mathcal{N}_s w \ dx \leq \langle f, \psi \rangle_{W^{-s,2}(\Omega, \Sigma_1), W^{s,2}(\Omega, \Sigma_1)}
  \end{equation}
  for all non-negative test functions defined on $\Sigma_2.$ Since $\psi \in \mathcal{D}(\Sigma_2)$, the right-hand-side in \eqref{eq:PP} vanishes. Therefore,
  \begin{align*}
      &\int_{\Sigma_2} \psi \mathcal{N}_s w \ dx \leq 0
  \end{align*}
  for all $\psi \in \mathcal{D}(\Sigma_2)$ with $\psi \geq 0$. As a result, 
  \begin{align*}
      & \mathcal{N}_s w \leq 0 \quad \mbox{in} \; \Sigma_2.
  \end{align*}

  Conversely, suppose that $w$ satisfies \eqref{EL1}. Then, 
  \begin{align*}
     \mathcal{E}(w, v -w) =& \langle (-\Delta)^sw, v-w \rangle_{W^{-s,2}(\Omega, \Sigma_1), W^{s,2}(\Omega, \Sigma_1)} + \int_{\Sigma_2} (v-w) \mathcal{N}_s w \ dx \\
     =  & \langle f, v - w \rangle_{W^{-s,2}(\Omega, \Sigma_1), W^{s,2}(\Omega, \Sigma_1)} + \int_{\{x \in \Sigma_2: w(x) < \varphi(x)\}} (v-w) \mathcal{N}_s w \ dx \\
     &+ \int_{\{x \in \Sigma_2: w(x) = \varphi}(x)\} (v - \varphi) \mathcal{N}_s w \ dx \\
     \geq &\langle f, v - w \rangle_{W^{-s,2}(\Omega, \Sigma_1), W^{s,2}(\Omega, \Sigma_1)}
  \end{align*}
  for all $v \in \mathcal{K}_0$. 
  
  It remains to show the last assertion of the proposition. Indeed, since $\mathcal N_sw=0$ in $\Sigma_2\cap\{u<\varphi\}$ and $u\le\varphi$ in $\Sigma_2$, it follows that $(u-\varphi)\mathcal N_sw=0$ in $\Sigma_2$. We have shown that the last two conditions in \eqref{EL1} implies \eqref{CC}. Now, assume that \eqref{CC} holds. This implies that  $u-\varphi=0$ in $\Sigma_2$ or $\mathcal N_sw=0$ in $\Sigma_2$. This trivially implies that $\mathcal N_sw=0$ in $\Sigma_2\cap\{u<\varphi\}$ and $u\le\varphi$ in $\Sigma_2$. The proof is finished.
  \end{proof}

\section{Penalization} \label{Penalization}

We now consider a variety of penalty formulations, whose purpose is to incorporate the constraint into our minimization problem and approximate our original formulation by a sequence of Fr\'echet differentiable functionals. We begin this section by analyzing a Moreau-Yosida type penalty formulation in $L^2(\Sigma_2)$, namely
%
        \begin{equation}\label{eq:PL2Sig2}
           \min_{W^{s,2}(\Omega,\Sigma_1)} J_{\epsilon}(w) := \frac{1}{2}\mathcal E(w,w) - \langle f, w \rangle_{W^{-s,2}(\Omega,\Sigma_1), W^{s,2}(\Omega,\Sigma_1)} + \frac{\epsilon^{-2}}{2} \int_{\Sigma_2} [(w - \varphi)^+]^2 \ dx 
        \end{equation}
where $\epsilon > 0$ is the penalty parameter and, for $u \in W^{s,2}(\Sigma_2)$, we denote the positive part of $u$ as $u^+ := \max \{u, 0\}.$ Further, we denote the negative part of $u$ to be $u^- := \min \{u,0\}$ and notice that $u = u^+ + u^-.$ 
        
        As before, the direct method of the calculus of variations ensures a unique minimizer to $J_{\epsilon}$, denoted by $w_{\epsilon}.$ 
        We have the following convergence result.
           \begin{proposition}\label{prop:mosco}
            For every $\epsilon > 0,$ there exists a solution $w_{\epsilon}$ to \eqref{eq:PL2Sig2}. Additionally, there exists a subsequence, that we still denote by $(w_{\epsilon})$,  of solutions that converges weakly to $w \in \mathcal K_0$, as $\epsilon\downarrow 0$,  so that $ J_{\epsilon} \xrightarrow[]{M} \Tilde{J}$ (in the sense of Mosco), as $\epsilon\downarrow 0$.
        \end{proposition}
        
        \begin{proof}
             For each $\epsilon>0$ and the minimization problem corresponding to $J_\epsilon$, consider the resulting sequence of solutions $(w_\epsilon)_{\epsilon>0}$.
             From the coercivity of $J$, there exists a constant $C > 0$ independent of $\epsilon$ so that $\lVert w_{\epsilon} \rVert_{W^{s,2}(\Omega, \Sigma_1)} \leq C.$ Then, there is a subsequence that we still denote by $(w_{{\epsilon}})_\epsilon$ that converges weakly to $w \in W^{s,2}(\Omega, \Sigma_1)$, as $\epsilon\downarrow 0.$ From the weak lower semi-continuity of $J,$ we have that
        \begin{equation}\label{eq:mosco}
            J(w) \leq \liminf_{\epsilon\downarrow 0} J(w_{\epsilon}) \leq \liminf_{\epsilon \downarrow 0} \left(J(w_{\epsilon}) + \frac{\epsilon^{-2}}{2} \int_{\Sigma_2} [(w_{\epsilon} - \varphi)^+]^2\;dx\right) = \liminf_{\epsilon\downarrow 0} J_{\epsilon}(w_{\epsilon}).
        \end{equation} 
        Further, we claim that the weak limit, $w$, belongs to $\mathcal K_0.$ Indeed, let $v \in \mathcal K_0$ be fixed.  Then, from the optimality of $w_\epsilon$ it follows that
        \begin{equation}\label{JJJ}
            J(w_{\epsilon}) + \frac{\epsilon^{-2}}{2} \int_{\Sigma_2} [(w_{\epsilon} - \varphi)^+]^2 \ dx  \leq 
            J(v) + \frac{\epsilon^{-2}}{2} \int_{\Sigma_2} [(v - \varphi)^+]^2\;dx = J(v),
        \end{equation}
        where we have used the fact that $(v-\varphi)^+=0$ in $\Sigma_2$. It follows from \eqref{JJJ}
     that 
        \begin{equation*}
            \int_{\Sigma_2} [(w_{\epsilon} - \varphi)^+]^2 \ dx  \leq 2\epsilon^2(J(v) - J(w_{\epsilon})) \leq C\epsilon^2.
        \end{equation*}
 Then, as a result of weak lower semi-continuity, we have that
        \begin{equation*}
            \int_{\Sigma_2} [(w - \varphi)^+]^2 \leq \liminf_{\epsilon\downarrow 0} \int_{\Sigma_2} [(w_{\epsilon} - \varphi)^+]^2 = 0
        \end{equation*}
        so that $w \in \mathcal K_0$, and the claim is proved.
        
    Further, from \eqref{eq:mosco} we have that
    $$
        J(w) \leq \liminf_{\epsilon\downarrow 0} J_{\epsilon}(w_{\epsilon}). 
    $$
        
Now, if $w \in W^{s,2}(\Omega, \Sigma_1),$ choosing the constant sequence $(w)_{\epsilon > 0}$ gives us that 
        
    $$
        \limsup_{\epsilon \downarrow 0} J_{\epsilon}(w) = J(w)
    $$ 
    and so we have the convergence in the sense of Mosco.  The proof is finished.
        \end{proof}            

        \begin{remark}
        {\em 
        We observe that generally,  $w_{\epsilon} \in W^{s,2}(\Omega, \Sigma_1)$ will fail to satisfy $w_{\epsilon} \leq \varphi$ in $\Sigma_2$. Therefore, it is necessary to estimate the error created by this penalization. Notice that Proposition~\ref{prop:mosco} establishes convergence. 
        }
        \end{remark}
    
    The next result shows a rate of convergence with respect to $\epsilon$.
  
        \begin{theorem}\label{thm:penal}
            The unique minimizer $w_{\epsilon} \in W^{s,2}(\Omega, \Sigma_1)$ of the penalized functional $J_{\epsilon}$ satisfies 
            \begin{equation}\label{pe-eq}
            C_{N,s} \lVert w - w_{\epsilon} \rVert^2_{W^{s,2}(\Omega, \Sigma_1)} + \frac {\epsilon^{-2}}{2}\lVert (w_{\epsilon} - \varphi)^+ \rVert_{L^2(\Sigma_2)}^2 \leq \epsilon^2 \lVert \mathcal{N}_s w\rVert^2_{L^2(\Sigma_2)},
            \end{equation}
            where we recall that the operator $\mathcal N_s$ is given in \eqref{NLND}.
        \end{theorem}
        
        \begin{proof}
        For any specified $\epsilon > 0$, the minimizer $w_\epsilon$ of \eqref{eq:PL2Sig2} satisfies 
        \begin{equation*}
            \mathcal{E}(w_{\epsilon}, v) + \epsilon^{-2}\int_{\Sigma_2} (w_{\epsilon}-\varphi)^+ v \ dx = \langle f, v \rangle_{W^{-s,2}(\Omega, \Sigma_1), W^{s,2}(\Omega, \Sigma_1)}  \mbox{ for all} \; v \in W^{s,2}(\Omega, \Sigma_1).
        \end{equation*}

        Applying the integration by parts formula given in \eqref{Int-Part}, we obtain that
        \begin{align}\label{eq: EL-PL2}
            \langle (-\Delta)^s w_{\epsilon}, v \rangle_{W^{-s,2}(\Omega, \Sigma_1), W^{s,2}(\Omega, \Sigma_1)} &+ \int_{\Sigma_2} v \mathcal{N}_s w_{\epsilon} \ dx + \epsilon^{-2}\int_{\Sigma_2} (w_{\epsilon}-\varphi)^+ v \ dx \notag\\ 
            =& \langle f, v \rangle_{W^{-s,2}(\Omega, \Sigma_1), W^{s,2}(\Omega, \Sigma_1)}
        \end{align}
        for all $v \in W^{s,2}(\Omega, \Sigma_1).$

        Further, from \eqref{Int-Part} and \eqref{eq: EL-PL2}, taking $v:=w-w_\epsilon$ as a  test function, we obtain that,
        \begin{align}
            & \notag \frac{C_{N,s}}{2}\lVert w - w_{\epsilon} \rVert^2_{W^{s,2}(\Omega, \Sigma_1)} \\
            \notag =& \langle (-\Delta)^s(w - w_{\epsilon}), w-w_{\epsilon} \rangle_{W^{-s,2}(\Omega, \Sigma_1), W^{s,2}(\Omega, \Sigma_1)} + \int_{\Sigma_2} (w - w_{\epsilon})\mathcal{N}_s(w-w_{\epsilon}) \ dx \\
             \notag = &\langle (-\Delta)^s w , w-w_{\epsilon} \rangle_{W^{-s,2}(\Omega, \Sigma_1), W^{s,2}(\Omega, \Sigma_1)} + \int_{\Sigma_2} (w - w_{\epsilon})\mathcal{N}_s w \ dx \\
            & +\epsilon^{-2} \int_{\Sigma_2} (w_{\epsilon} - \varphi)^+ (w - w_{\epsilon}) \ dx - \langle f, w - w_{\epsilon} \rangle_{W^{-s,2}(\Omega, \Sigma_1), W^{s,2}(\Omega, \Sigma_1)}. \label{P1}
        \end{align}
The first two identities in \eqref{EL1} implies that
        \begin{equation}\label{bb-1}
        \langle (-\Delta)^s w , w-w_{\epsilon} \rangle_{W^{-s,2}(\Omega, \Sigma_1), W^{s,2}(\Omega, \Sigma_1)}   - \langle f, w - w_{\epsilon} \rangle_{W^{-s,2}(\Omega, \Sigma_1), W^{s,2}(\Omega, \Sigma_1)}=0.
        \end{equation}
       Since
        \begin{equation*}
            \int_{\Sigma_2}(w-\varphi)\mathcal N_sw\;dx=0,
        \end{equation*}
        then using \eqref{bb-1} we see that \eqref{P1} becomes 
        \begin{align}\label{P2}
        \frac{C_{N,s}}{2}\lVert w - w_{\epsilon} \rVert^2_{W^{s,2}(\Omega, \Sigma_1)}=     -\int_{\Sigma_2} (w_{\epsilon} - \varphi) \mathcal{N}_s w \ dx + \epsilon^{-2}\int_{\Sigma_2} (w_{\epsilon} - \varphi)^+(w-w_{\epsilon}) \ dx.
        \end{align}

       Since $w \leq \varphi$ in $\Sigma_2$, we have that
        \begin{align}
            & \notag \epsilon^{-2} \int_{\Sigma_2} (w_{\epsilon} - \varphi)^+(w-w_{\epsilon}) \ dx \\
            & \notag = -\epsilon^{-2}\int_{\Sigma_2}(w_{\epsilon} - \varphi)^+(w_{\epsilon} - \varphi) \ dx - \epsilon^{-2}\int_{\Sigma_2}(w_{\epsilon} - \varphi)^+(\varphi - w) \ dx \\
            & \leq -\epsilon^{-2} \int_{\Sigma_2} (w_{\epsilon} - \varphi)^+(w_{\epsilon} - \varphi) \ dx. \label{P3}
        \end{align}

        Since $(w_{\epsilon} - \varphi)^+(w_{\epsilon} - \varphi)^- = 0,$ the right-hand-side in \eqref{P3} becomes,
        \begin{align}
            & \notag -\epsilon^{-2} \int_{\Sigma_2}[(w_{\epsilon} - \varphi)^+]^2 \ dx - \epsilon^{-2} \int_{\Sigma_2} (w_{\epsilon} - \varphi)^+(w_{\epsilon} - \varphi)^- \ dx \\
            & = -\epsilon^{-2} \int_{\Sigma_2}[(w_{\epsilon} - \varphi)^+]^2 \ dx. \label{P4}
        \end{align}

        Combining \eqref{P1}-\eqref{P4} and recalling that $\mathcal{N}_sw \leq 0$ in $\Sigma_2$, we get
        \begin{align}
             \notag \frac{C_{N,s}}{2}\lVert w - w_{\epsilon} \rVert^2_{W^{s,2}(\Omega, \Sigma_1)} \leq& -\int_{\Sigma_2} (w_{\epsilon} - \varphi) \mathcal{N}_s w \ dx - \epsilon^{-2} \int_{\Sigma_2} [(w_{\epsilon} - \varphi)^+]^2 \ dx \\
             \notag = &- \int_{\Sigma_2} (w_{\epsilon} - \varphi)^+ \mathcal{N}_s w \ dx - \int_{\Sigma_2} (w_{\epsilon} - \varphi)^- \mathcal{N}_s w \ dx \notag\\
            &- \epsilon^{-2} \int_{\Sigma_2} [(w_{\epsilon} - \varphi)^+]^2 \ dx\notag \\
            \leq& - \int_{\Sigma_2} (w_{\epsilon} - \varphi)^+ \mathcal{N}_s w \ dx - \epsilon^{-2} \int_{\Sigma_2} [(w_{\epsilon} - \varphi)^+]^2 \ dx. \label{P5}
        \end{align}
        This implies that
        \begin{align}
            \notag\frac{C_{N,s}}{2}\lVert w - w_{\epsilon} \rVert^2_{W^{s,2}(\Omega, \Sigma_1)} +
            \epsilon^{-2} \| (w_{\epsilon} - \varphi)^+\|_{L^2(\Sigma_2)}^2 \ dx
            \leq - \int_{\Sigma_2} (w_{\epsilon} - \varphi)^+ \mathcal{N}_s w \ dx \\
            \le \frac{\epsilon^{-2}}{2}\lVert (w_{\epsilon} - \varphi)^+ \rVert^2_{L^2(\Sigma_2)} + \frac{\epsilon^2}{2}\lVert \mathcal{N}_s w \rVert^2_{L^2(\Sigma_2)}
            \label{eq:PP0}
        \end{align}
        where in the last step we have used the H\"older's inequality and the generalized Young's inequality. We have shown \eqref{pe-eq} and        
        the proof  is finished.
        \end{proof}
                
         In the above penalized formulation (cf.~\eqref{eq:PL2Sig2}) and subsequently in Theorem~\ref{thm:penal}, we consider a penalization in $L^2(\Sigma_2)$ norm. Next, we instead consider a penalization in $W^{s,2}(\Sigma_2)$-norm. The advantage being that the optimality conditions for the penalized problem can be directly related to the original optimality system \eqref{eq: KKT}.
        
        Motivated by Kikuchi and Oden \cite[Chapter 3]{Kik_Oden_Thin_Obst_Penalization}, we look at a penalty functional of the form
        \begin{equation}\label{eq413}
            J_{\xi}(w) = \frac{1}{2}\mathcal{E}(w,w) - \langle f,w \rangle_{W^{-s,2}(\Omega, \Sigma_1), W^{s,2}(\Omega, \Sigma_1)} +\frac{\xi^{-1}}{2}\lVert (w- \varphi)^+\rVert_{W^{s,2}(\Sigma_2)}^2.
        \end{equation}

           From the strict convexity of $J_{\xi}$ as well as the direct method of the calculus of variations we know that $J_{\xi}$ has a unique minimizer, which we denote by $w_{\xi}.$ Furthermore, $w_{\xi}$ satisfies the optimality conditions
    \begin{equation} \label{Pen2EL}
        \mathcal{E}(w_{\xi}, v) + \frac{1}{\xi}((w_{\xi} - \varphi)^+, v)_{W^{s,2}(\Sigma_2)} = \langle f, v \rangle_{W^{-s,2}(\Omega, \Sigma_1), W^{s,2}(\Omega, \Sigma_1)}  
    \end{equation}
    for all $v \in W^{s,2}(\Omega, \Sigma_1).$
        As we will see, considering such a penalty functional gives us a method to relate our penalized problem back to the optimality system in \eqref{eq: KKT}. More precisely, we have the following result. 

\begin{theorem}\label{thm:zeta}
        Suppose that $w \in \mathcal{K}_0$ and $\lambda \in W^{-s,2}(\Sigma_2)\subset W^{-s,2}(\Omega,\Sigma_1)$ satisfy \eqref{eq: KKT}. Additionally, suppose that $w_{\xi}$ minimizes $J_{\xi}$ for a given $\xi >0$.  Let $\lambda_\xi$ be the unique element in $W^{-s,2}(\Sigma_2)$ satisfying (by the Riesz representation theorem)
        \begin{equation}\label{mm-e}
            \langle \lambda_\xi,v\rangle_{W^{-s,2}(\Sigma_2),W^{s,2}(\Sigma_2)}=\left(\frac{1}{\xi}(w_{\xi} - \varphi)^+,v\right)_{W^{s,2}(\Sigma_2)}\;\;\mbox{ for all } v\in W^{s,2}(\Sigma_2).
        \end{equation}
        Then, there is a constant $C=C(N,s,\Omega)>0$ such that 
        \begin{equation}\label{eq414}
            \lVert w_{\xi} - w\rVert_{W^{s,2}(\Omega,\Sigma_1)} \leq C\xi\lVert \lambda \rVert_{W^{-s,2}(\Sigma_2)}
        \end{equation}
     and
     \begin{equation}\label{eq415}
            \lVert \lambda_{\xi} - \lambda\rVert_{W^{-s,2}(\Sigma_2)} \leq C\xi\lVert \lambda \rVert_{W^{-s,2}(\Sigma_2)}.
        \end{equation}
        In particular, we have that $w_{\xi} \to w$ in $W^{s,2}(\Om,\Sigma_1)$ and $\lambda_{\xi} \to  \lambda$ in $W^{-s,2}(\Sigma_2)$, as $\xi\downarrow 0$. 

        In addition,  there is a constant $C>0$ independent of $\xi$ such that
   \begin{equation}\label{eq-xi}
       \|(w_{\xi}-\varphi)^+\|_{W^{s,2}(\Sigma_2)}\le C\xi.
   \end{equation}
         \end{theorem}                
        \begin{proof}
           It follows from \eqref{eq: KKT}, \eqref{Pen2EL} and \eqref{mm-e} that 
            \begin{align}\label{PEQ1}
               \mathcal{E}(w_{\xi} - w, v) = \langle \lambda_\xi - \lambda, v\rangle_{W^{-s,2}(\Sigma_2),W^{s,2}(\Sigma_2)} 
            \end{align}
            for all $v \in W^{s,2}(\Omega, \Sigma_1).$  Taking $v:=w_\xi-w$ in \eqref{PEQ1} yields
            \begin{equation}\label{PEQ1-2}
                \mathcal E(w_\xi-w,w_\xi-w)= \langle \lambda_\xi - \lambda, w_\xi-w\rangle_{W^{-s,2}(\Sigma_2),W^{s,2}(\Sigma_2)} .
            \end{equation}
            Since by \eqref{eq: KKT}
            $$\langle \lambda , w-\varphi\rangle_{W^{-s,2}(\Sigma_2),W^{s,2}(\Sigma_2)}=0,$$ 
            it follows that
            \begin{align}\label{mwj}
              \langle \lambda - \lambda_\xi, w-w_\xi\rangle_{W^{-s,2}(\Sigma_2),W^{s,2}(\Sigma_2)}=& \langle \lambda - \lambda_\xi, w-\varphi+\varphi-w_\xi\rangle_{W^{-s,2}(\Sigma_2),W^{s,2}(\Sigma_2)}\notag\\
              =&
              \langle -\lambda_\xi , w-\varphi\rangle_{W^{-s,2}(\Sigma_2),W^{s,2}(\Sigma_2)}\notag\\
              &-
              \langle \lambda - \lambda_\xi, w_\xi-\varphi\rangle_{W^{-s,2}(\Sigma_2),W^{s,2}(\Sigma_2)}.
            \end{align}
 It follows from the abstract result given in \cite[Section 3.3]{Kik_Oden_Thin_Obst_Penalization}, that for $f\in W^{s,2}(\Sigma_2)$, $f^+$ is characterized by the variational inequality 
            \begin{equation}\label{projection_VI}
                (f^+ - f, g - f^+)_{W^{s,2}(\Sigma_2)} \geq 0,
            \end{equation}
            for all $g\in\mathcal K_0^+ := \{g \in W^{s,2}(\Sigma_2): g \geq 0\}$.
            In addition, since $f^+\in\mathcal K_0^+$ we can easily deduce that 
            \begin{equation*}
            (f^-,f^+)_{W^{s,2}(\Sigma_2)}=0,\;  (g^-,f^+)_{W^{s,2}(\Sigma_2)}\le 0,\;\;  (f^+,f^+)_{W^{s,2}(\Sigma_2)}\ge 0,
            \end{equation*}
            for all $f$, $g\in W^{s,2}(\Sigma_2)$. Notice also that from the definition of $\lambda_\xi$, we have that 
            \begin{equation*}
            \langle\lambda_\xi-\lambda,\xi^{-1}(w_\xi-\varphi)^+\rangle_{W^{-s,2}(\Sigma_2),W^{s,2}(\Sigma_2)}=\left(\lambda_\xi-\lambda,\lambda_\xi\right)_{W^{-s,2}(\Sigma_2)}.
            \end{equation*}
            Using all these facts and \eqref{mm-e}, we get from \eqref{mwj} that 
            \begin{align*}
               \langle \lambda - \lambda_\xi, w-w_\xi\rangle_{W^{-s,2}(\Sigma_2),W^{s,2}(\Sigma_2)}
               =&\langle \lambda - \lambda_\xi, w-\varphi+\varphi-w_\xi\rangle_{W^{-s,2}(\Sigma_2),W^{s,2}(\Sigma_2)}\\
               =&-\langle \lambda - \lambda_\xi, w_\xi-\varphi\rangle_{W^{-s,2}(\Sigma_2),W^{s,2}(\Sigma_2)}\\
               &+\langle \lambda - \lambda_\xi, w-\varphi\rangle_{W^{-s,2}(\Sigma_2),W^{s,2}(\Sigma_2)}\\
                =& -\langle \lambda - \lambda_\xi, w_\xi-\varphi\rangle_{W^{-s,2}(\Sigma_2),W^{s,2}(\Sigma_2)}\\
               &+\langle \lambda, w-\varphi\rangle_{W^{-s,2}(\Sigma_2),W^{s,2}(\Sigma_2)}\\
               &-\langle \lambda_\xi, w-\varphi\rangle_{W^{-s,2}(\Sigma_2),W^{s,2}(\Sigma_2)}\\
                =&-\langle \lambda - \lambda_\xi, w_\xi-\varphi\rangle_{W^{-s,2}(\Sigma_2),W^{s,2}(\Sigma_2)}\\
                &+\langle \lambda, w-\varphi\rangle_{W^{-s,2}(\Sigma_2),W^{s,2}(\Sigma_2)}\\
                &-\left(\frac{1}{\xi}(w_\xi-\varphi)^+,w-\varphi\right)_{W^{s,2}(\Sigma_2)}.
                \end{align*}
        Since, $(w-\varphi)^+=0$ in $\Sigma_2$, we have that
        \begin{align*}
        &\left(\frac{1}{\xi}(w_\xi-\varphi)^+,w-\varphi\right)_{W^{s,2}(\Sigma_2)}\\
        =& \left(\frac{1}{\xi}(w_\xi-\varphi)^+,(w-\varphi)^+\right)_{W^{s,2}(\Sigma_2)}+
        \left(\frac{1}{\xi}(w_\xi-\varphi)^+,(w-\varphi)^-\right)_{W^{s,2}(\Sigma_2)}\\
        =&\left(\frac{1}{\xi}(w_\xi-\varphi)^+,(w-\varphi)^-\right)_{W^{s,2}(\Sigma_2)}\\
        \le& 0.
        \end{align*}
        Notice also that
        $$\langle \lambda_\xi, (w_\xi-\varphi)^-\rangle_{W^{-s,2}(\Sigma_2),W^{s,2}(\Sigma_2)}=0$$
        and
        $$\langle \lambda , (w_\xi-\varphi)^-\rangle_{W^{-s,2}(\Sigma_2),W^{s,2}(\Sigma_2)}\le 0.$$
        Using all the above facts we can deduce that 
        \begin{align*}
         \langle \lambda - \lambda_\xi, w-w_\xi\rangle_{W^{-s,2}(\Sigma_2),W^{s,2}(\Sigma_2)}       
               \le & \langle \lambda_\xi - \lambda, w_\xi-\varphi\rangle_{W^{-s,2}(\Sigma_2),W^{s,2}(\Sigma_2)}\\
               \le & \langle \lambda_\xi - \lambda, (w_\xi-\varphi)^+\rangle_{W^{-s,2}(\Sigma_2),W^{s,2}(\Sigma_2)}\\
               =&-\xi  \left(\lambda_\xi-\lambda,\lambda_\xi\right)_{W^{-s,2}(\Sigma_2)}\\
               =&-\xi\left(\lambda_\xi-\lambda,\lambda_\xi-\lambda+\lambda\right)_{W^{-s,2}(\Sigma_2)}\\
               =&-\xi\|\lambda_\xi-\lambda\|_{W^{-s,2}(\Sigma_2)}^2
               -\xi\left(\lambda_\xi-\lambda,\lambda\right)_{W^{-s,2}(\Sigma_2)}\\
                \le &\xi \|\lambda-\lambda_\xi\|_{W^{-s,2}(\Sigma_2)}\|\lambda\|_{W^{-s,2}(\Sigma_2)}.
            \end{align*}
            This latter estimate together with \eqref{PEQ1-2} and the coercivity of the bilinear form $\mathcal E$ yield that there is a constant $C=C(N,s,\Omega)>0$ such that
            \begin{align}\label{ss1}
                \|w_\xi-w\|_{W^{s,2}(\Omega,\Sigma_1)}^2\le C\xi\|\lambda-\lambda_\xi\|_{W^{-s,2}(\Sigma_2)}\|\lambda\|_{W^{-s,2}(\Sigma_2)}.
            \end{align}
            On the other hand, it follows from \eqref{PEQ1} that there is a constant $C>0$ such that
            \begin{align}\label{ss2}
             \|\lambda-\lambda_\xi\|_{W^{-s,2}(\Sigma_2)}\le C\|w_\xi-w\|_{W^{s,2}(\Omega,\Sigma_1)}.
            \end{align}
           Combining \eqref{ss1}-\eqref{ss2} we obtain that 
            \begin{align*}
             \|w_\xi-w\|_{W^{s,2}(\Omega,\Sigma_1)}\le C\xi \|\lambda\|_{W^{-s,2}(\Sigma_2)}
            \end{align*}
            and we have shown \eqref{eq414}. Combining \eqref{eq414} and \eqref{ss2} we get \eqref{eq415}. 
            Finally, the last assertion follows from \eqref{ss2}. 
            
          It remains to prove \eqref{eq-xi}.  Indeed, taking $v:=(w_{\xi}-\varphi)^+\in W^{s,2}(\Sigma_2) $ as a test function in \eqref{mm-e} we obtain the following estimate:
   \begin{align*}
       \|(w_{\xi}-\varphi)^+\|_{W^{s,2}(\Sigma_2)}^2=&\xi\langle \lambda_\xi,(w_{\xi}-\varphi)^+\rangle_{W^{-s,2}(\Sigma_2), W^{s,2}(\Sigma_2)}\\
       \le & \xi\|\lambda_\xi\|_{W^{-s,2}(\Sigma_2)}\|(w_{\xi}-\varphi)^+\|_{ W^{s,2}(\Sigma_2)}.
   \end{align*}
   Since $\lambda_\xi\to\lambda$ in $W^{-s,2}(\Sigma_2)$, as $\xi\to 0$, it follows that the sequence $(\lambda_\xi)$ is bounded in $W^{-s,2}(\Sigma_2)$. This fact together with the previous estimate gives \eqref{eq-xi}. The proof is finished.
            \end{proof}

\bibliographystyle{plain}
\bibliography{refs}

\end{document}